\documentclass[a4paper,11pt]{amsart}

\usepackage[utf8]{inputenc}
\usepackage[all,cmtip]{xy}
\usepackage[english]{babel}
\usepackage{mathtools}
\usepackage{graphicx}
\usepackage{multicol}

\usepackage{amsmath,amssymb,amsthm,amsrefs,mathrsfs}
\usepackage{xfrac}
\usepackage{pifont}

\usepackage{array,tabularx}
\usepackage{xcolor}

\usepackage{tikz,tikz-cd}
\usetikzlibrary{positioning, shapes.geometric}
\usetikzlibrary{arrows,arrows.meta}
\usetikzlibrary{calc}
\tikzcdset{arrow style=tikz, diagrams={>={Computer Modern Rightarrow[width=5.5pt,length=3pt]}}}

\usepackage{enumitem}
\usepackage{cleveref}

\usepackage[retainorgcmds]{IEEEtrantools}

\theoremstyle{plain}
\newtheorem{thm}{Theorem}[section] 
\newtheorem*{thm*}{Theorem}
\newtheorem*{mainthm}{Main Theorem}
\newtheorem{prop}[thm]{Proposition}
\newtheorem{cor}[thm]{Corollary}
\newtheorem{lem}[thm]{Lemma}
\newtheorem{letlem}{Lemma}

\theoremstyle{definition}

\newtheorem{exa}[thm]{Example}
\newtheorem{rem}[thm]{Remark}

\newcommand*{\myproofname}{Proof of Theorem \ref{1thm:inertness}}
\newenvironment{thmproof}[1][\myproofname]{\begin{proof}[#1]}{\end{proof}}

\newcommand*{\myproofnames}{Proof of Proposition \ref{prop:pdhmlgy}}
\newenvironment{propproof}[1][\myproofnames]{\begin{proof}[#1]}{\end{proof}}
  
\newcommand{\N}{\mathbb{N}}
\newcommand{\Z}{\mathbb{Z}}
\newcommand{\Q}{\mathbb{Q}}

\newcommand{\C}{\mathbb{C}}

\usepackage{libertine}
\usepackage[libertine]{newtxmath}
\renewcommand{\mathbb}{\varmathbb}

\title{A Homotopy Theoretic Analogue to a Theorem of Wall}
\author{Sebastian Chenery}
\address{Mathematical Sciences, University of Southampton, Southampton SO17 1BJ, United Kingdom}
\email{s.d.chenery@soton.ac.uk}


\begin{document}

\maketitle

\begin{abstract}
It is a well-known result of C.T.C. Wall's that one may decompose a simply connected 6-manifold as a connected sum of two simpler manifolds. Recent work of Beben and Theriault on decomposing based loop spaces of highly connected Poincar\'e Duality complexes has yielded new methods for analysing the homotopy theory of manifolds. In this paper we will expand upon these methods, which we will then apply to prove a higher dimensional homotopy theoretic analogue to Wall’s Theorem.


\end{abstract}

\section{Introduction}

When studying the algebraic topology of manifolds, it is natural to begin with those manifolds which are highly connected. Indeed, such study has a rich history: Milnor recounts in \cite{milnorclass} that during the 1950s he was concerned with $(n-1)$-connected $2n$-manifolds, and Ishimoto classified $\pi$-manifolds of this type in \cite{ishimoto69}. Another active author in this area was C.T.C. Wall, who in 
\cite{wall6mfld} sought a classification of simply connected 6-manifolds. Using methods of differential topology and surgery theory, Wall gave the following theorem.

\begin{thm*}[Wall]
Let $M$ be a closed, smooth, simply connected 6-manifold. Then there is a diffeomorphism \[M\cong M_1\#M_2\] where $M_1$ is a connected sum of finitely many copies of $S^3\times S^3$ and $H_3(M_2)$ is finite.
\end{thm*}

Generalising this theorem to higher dimensions leads one to consider decomposing $(n-2)$-connected $2n$-manifolds into constituent parts via the operation of connected sums. In the past decade there has been much activity studying highly connected manifolds via homotopy theory, notably by Beben and Theriault \cites{bt1, bt2}, in which they consider the based loop spaces of $(n-1)$-connected $2n$-manifolds. More recently, work of Huang \cite{huanggauge} incorporated a study of torsion free $(n-2)$-connected $2n$-manifolds with vanishing cohomology in dimension $n$. These papers do not explore connected sums directly, but instead give decompositions of loop spaces as products of other spaces. This leads us to the content of this paper: drawing on this recent work, and making use of known results for based loop spaces of certain complexes, we give the following homotopy theoretic analogue to Wall's Theorem, which we prove in Theorem \ref{thm:wallanalogue}. 

\begin{mainthm}
Let $n>3$ be an integer such that $n\not\in\lbrace4,8\rbrace$, and let $M$ be a \((n-2)\)-connected \(2n\)-dimensional Poincar\'e Duality complex with \(rank(H_n(M))=d>1\). Then there exists a homotopy equivalence \[\Omega M\simeq\Omega(M_1\#M_2\#M_3)\] where
\begin{itemize}
    \item[(i)] \(M_1\) is an \((n-1)\)-connected \(2n\)-dimensional Poincar\'e Duality complex, with \(rank(H_n(M_1))=d\);
    \item[(ii)] \(M_2\) is a connected sum of finitely many copies of \(S^{n-1}\times S^{n+1}\) and;
    \item[(iii)] $M_3$ is a \(CW\)-complex with $H_n(M_3)$ finite.
\end{itemize}
\end{mainthm}

This has result has several implications, notably that the homotopy groups of $M$ are determined by those of the connected sum $M_1\#M_2\#M_3$. More deeply, it also implies that in almost all cases, the homotopy groups of a $(n-2)$-connected $2n$-manifold are rationally hyperbolic (see Corollary \ref{cor:hyp}). Note however that we specifically exclude the case of a simply connected and \(6\)-dimensional Poincar\'e Duality complex - the techniques required for decomposing such complexes are very different to those discussed in this paper, see for example \cites{cutlerso, huang6over4}. Furthermore, it also bears mentioning that the complexes \(M_1\) and \(M_3\) in the above theorem may in some cases have the homotopy type of a manifold: when the integer \(d\) is even, we may take \(M_1\) to be a connected sum of \(\frac{d}{2}\)-many copies of \(S^n\times S^n\), and for \(M_3\) it depends on the total surgery obstruction of Ranicki \cite{ranickibook}. 

Our Main Theorem is by no means the first higher dimensional analogue to Wall's Theorem. Tamura gave decomposition results for closed, oriented, torsion free $(n-2)$-connected differentiable $2n$-manifolds (for certain congruence classes of \(n\) modulo 8) with vanishing \(n^{th}\) homology group \cite{tamura2} . Later, in the 70s, Ishimoto was able to expand on this, giving a partial analogue to Wall's Theorem for $(n-2)$-connected $2n$-manifolds with torsion free homology, using results about parallelisability \cite{ishimoto}. Indeed, \cite{ishimoto}*{Theorem 4} shows that a unique connected sum decomposition (up to reordering the summands) always exists for these manifolds, and a partial answer to the consequent classification problem is subsequently developed - see for example \cite{ishimoto}*{Theorem 7}. In analogue to Wall's Theorem, both \cite{tamura2} and \cite{ishimoto} detect copies of \((n-1)\)-connected \(2n\)-dimensional manifolds as summands in their connected sum decompositions. They also work hard to gain more control over the diffeomorphism type of the space analogous to Wall's \(M_2\). Further work from the last century, including but not limited to \cites{ishimoto69, fang}, continued this trend of using geometric and differential methods developed from those of Wall in order to provide higher dimensional analogues.

The methods used in this paper differ greatly from those of the past authors mentioned above. Indeed, we focus on decomposing Poincar\'e Duality complexes (of which closed, smooth, simply connected manifolds are a subclass), and our restrictions are far milder: we do not need to make assumptions about parallisability or restrict to the case when homology is torsion free. Though the price we pay is to sacrifice geometric precision by passing to based loop spaces, we still recover useful homotopy theoretic information, and demonstrate the value in considering such decomposition problems from this point of view. 

The structure of this paper is as follows. In the first section, we establish some modifications to a construction of Theriault \cite{t20}*{Section 8} in order to prove Theorem \ref{1thm:inertness}, which underpins much of what follows. In particular, this is used in Section 3 to give an important fact in Lemma \ref{lem:inertness}, and then to give further results in the context of the homotopy theory of Poincar\'e Duality complexes. We give several examples, as well as a new proof of \cite{t20}*{Theorem 9.1(b)-(c)}. The key theorem of this section is Theorem \ref{thm:PDconn}, which gives a general framework for our analysis. The next section is an aside into the skeletal structure of $(n-2)$-connected $2n$-dimensional Poincar\'e Duality complexes, which we discuss in order to prove the Main Theorem. The titular homotopy theoretic analogue is proved in Section \ref{sec:wall}, which is given by applying the methods developed throughout the preceding sections. 

The author would like to thank their PhD supervisor, Stephen Theriault, for the many illuminating discussions throughout the preparation of this paper.

\section{A Preliminary Construction}

In this section, we shall make some small changes to a construction of Theriault from \cite{t20}*{Section 8} in order to prove a slightly more general result, which forms the basis of what is to come. Let us first establish some notation. Unless otherwise stated, all spaces are assumed to be simply connected. For a homotopy cofibration \[A\xrightarrow{f}B\xrightarrow{j}C\] the map \(f\) is called \textit{inert} if \(\Omega j\) has a right homotopy inverse. Second, for a wedge of spaces \(\bigvee_{i=1}^n X_i\), let \(p_j:\bigvee_{i=1}^n X_i\rightarrow X_j\) denote the pinch map to the \(j^{th}\) summand. Every \(p_j\) has a right homotopy inverse, given by inclusion of the \(j^{th}\) wedge summand. Our focus in this section will be on analysing homotopy cofibrations of the form \[\Sigma A \xrightarrow{f} X\vee Y\xrightarrow{q} C\] and the key to our considerations will be the following homotopy commutative diagram of homotopy cofibrations \begin{equation}\label{1dgm:setup1}
    \begin{tikzcd}[row sep=1.5em, column sep = 1.5em]
        && Y \arrow[rr, equal] \arrow[dd, hook] && Y \arrow[dd] \\
        &&&&& \\
        \Sigma A \arrow[rr, "f"] \arrow[dd, equal] && X\vee Y  \arrow[rr, "q"] \arrow[dd, "p_1"] && C \arrow[dd, "\varphi"] \\
        &&&&& \\
        \Sigma A \arrow[rr, "p_1\circ f"] && X \arrow[rr, "j"] && M
    \end{tikzcd}
\end{equation}
where the bottom-right square is a homotopy pushout. Our goal is to prove the following theorem.  

\begin{thm}\label{1thm:inertness}
Consider Diagram (\ref{1dgm:setup1}). If the map \(\Omega j\) has a right homotopy inverse, then so do \(\Omega \varphi\) and \(\Omega q\). In particular, if the composite \(p_1\circ f\) is inert, then so is \(f\).
\end{thm}

We will be following the method set out by Theriault in \cite{t20}*{Section 8} very closely, though with some alterations. We include much of the argument here for the sake of transparency, though we will not include sections where we argue identically to Theriault; we will instead make this clear and give precise references for where the arguments can be found. We therefore suggest that this section be read in tandem with \cite{t20}*{Section 8}.

First, recall that for two path connected and based spaces $X$ and $Y$, the \textit{(left) half-smash} of $X$ and $Y$ is the quotient space \[X\ltimes Y=(X\times Y)/(X\times y_0)\] where $y_0$ denotes the basepoint of $Y$. We begin by stating one of the main theorems of \cite{t20}.

\begin{thm}[Theriault] \label{1thm:data}
Suppose there exists a homotopy commutative diagram
\[
    \begin{tikzcd}[row sep=1.5em, column sep = 1.5em]
        && E \arrow[rr, "\alpha"] \arrow[dd] && E' \arrow[dd] \\
        &&&&& \\
        \Sigma A \arrow[rr, "f"] && B \arrow[rr] \arrow[dd, "h"] && C \arrow[dd, "h'"] \\
        &&&&& \\
        && Z \arrow[rr, equal] && Z
    \end{tikzcd}
\]
where the middle and right columns are homotopy fibrations, the map \(\alpha\) is an induced map of fibres and the middle row is a homotopy cofibration. If $\Omega h$ has a right homotopy inverse, then there exists a homotopy cofibration \[\Omega Z\ltimes \Sigma A\xrightarrow{\theta}E\rightarrow E'\] for some map \(\theta\). \hfill $\square$
\end{thm}

Note the special case in which \(C=Z\) and \(h'\) is the identiy map, which implies that \(E'\) is contractible and therefore that \(\theta\) is a homotopy equivalence. This gives the following corollary (a version of which the reader will also find in \cite{bt2}*{Proposition 3.5}, where it first appeared). Note that the need for the suspension \(\Sigma A\) is dropped.

\begin{cor}[Beben-Theriault]\label{1cor:splitfib} 
Suppose there is a homotopy cofibration \[A\xrightarrow{f}B\xrightarrow{h}C\] such that the map $\Omega h$ has a right homotopy inverse. Then there exists a homotopy fibration \[\Omega C\ltimes A\rightarrow B\xrightarrow{h}C.\] Moreover, this homotopy fibration splits after looping, so there is a homotopy equivalence $\Omega B\simeq \Omega C\times\Omega(\Omega C\ltimes A).$ \hfill $\square$
\end{cor}

Returning to the situation of Diagram (\ref{1dgm:setup1}), since $p_1\circ f$ is inert, the map \(\Omega j\) has a right homotopy inverse. Let $F$ be the homotopy fibre of \(j\). Corollary \ref{1cor:splitfib} applies to the homotopy cofibration in the bottom row of (\ref{1dgm:setup1}), which implies that there is a homotopy equivalence \(F\simeq\Omega M\ltimes \Sigma A\). Equivalently, there is a homotopy cofibration \[\Omega M\ltimes\Sigma A\xrightarrow{\theta} F\rightarrow \ast.\] Now, let \(s\) denote a right homotopy inverse of \(\Omega j\), and \(t\) that of \(\Omega p_1\). Then the composite \(\Omega q\circ t\circ s\) is a right homotopy inverse for \(\Omega \varphi\). Let \(h=j\circ p_1\) and let \(E\) and \(E'\) denote the homotopy fibres of \(h\) and \(\varphi\), respectively. We have the following homotopy commutative diagram
\begin{equation}\label{1dgm:data2} 
    \begin{tikzcd}[row sep=1.5em, column sep = 1.5em]
        && E \arrow[rr, "\alpha"] \arrow[dd] && E' \arrow[dd] \\
        &&&&& \\
        \Sigma A \arrow[rr, "f"] && X\vee Y \arrow[rr, "q"] \arrow[dd, "h"] && C \arrow[dd, "\varphi"] \\
        &&&&& \\
        && M \arrow[rr, equal] && M
    \end{tikzcd}
\end{equation}
where the middle and right columns are homotopy fibrations, the map \(\alpha\) is an induced map of fibres and the middle row is a homotopy cofibration. Therefore, by Theorem \ref{1thm:data}, there exists a homotopy cofibration \begin{equation}\label{1cofib1}    
    \Omega M\ltimes\Sigma A\xrightarrow{\theta_f} E\xrightarrow{\alpha} E'. 
\end{equation}
Moreover, the right homotopy inverse for \(\Omega \varphi\) enables us to apply Corollary \ref{1cor:splitfib} to the right-most column of (\ref{1dgm:setup1}), so we have homotopy equivalences \[E'\simeq\Omega M\ltimes Y\text{\; and \;}\Omega C\simeq\Omega M\times\Omega(\Omega M\ltimes Y).\] The proof strategy for Theorem \ref{1thm:inertness} will be as follows: we wish to gain more control over the homotopy class of the first equivalence, and use this knowledge to deduce further facts about the second. The next step is to consider the homotopy fibration diagram
\begin{equation}\label{1dgm:fib1}
    \begin{tikzcd}[row sep=1.5em, column sep = 1.5em]
        E \arrow[rr, dashed, "l"] \arrow[dd] && F \arrow[dd] \\
        &&& \\
        X\vee Y \arrow[rr, "p_1"] \arrow[dd, "h"] && X \arrow[dd, "j"] \\
        &&& \\
        M \arrow[rr, equal] && M.
    \end{tikzcd}
\end{equation}
The bottom square of (\ref{1dgm:fib1}) is commutative by definition of the map \(h\), so the induced map of fibres \(l\) exists. In \cite{t20}*{Remark 2.7}, a naturality condition is given for Theorem \ref{1thm:data}, which is satisfied in by virtue of (\ref{1dgm:fib1}). Thus there is a homotopy cofibration diagram 
\begin{equation}\label{1dgm:cofib1}
    \begin{tikzcd}[row sep=1.5em, column sep = 1.5em]
        \Omega M \ltimes \Sigma A \arrow[rr, "\theta_f"] \arrow[dd, equal] && E  \arrow[rr, "\alpha"] \arrow[dd, "l"] && E' \arrow[dd] \\
        &&&&& \\
        \Omega M \ltimes \Sigma A \arrow[rr, "\theta"] && F \arrow[rr] && *.
    \end{tikzcd}
\end{equation}
Note that since \(\theta\) is a homotopy equivalence, (\ref{1dgm:cofib1}) implies that the map \(\theta_f\) always has a left homotopy inverse. Moreover, observe also that in constructing the above we only considered with the behaviour of $f$ when restricted to $X$. We record this in the lemma below, for ease of reference.

\begin{lem}\label{1lem:thetaleftinv}
With the set-up of Diagram (\ref{1dgm:cofib1}) above, the map \(\theta_f\) has a left homotopy inverse, regardless of the homotopy class of the composite \(f\) when restricted away from \(X\). \hfill $\square$
\end{lem}

This enables us to switch focus for the time being, and consider the special case in which \(C\simeq M\vee Y\). Diagram (\ref{1dgm:setup1}) now becomes the homotopy cofibration diagram
\begin{equation}\label{1dgm:cofib2}
    \begin{tikzcd}[row sep=1.5em, column sep = 1.5em]
        && Y \arrow[rr, equal] \arrow[dd, hook] && Y \arrow[dd, hook] \\ 
        &&&&& \\
        \Sigma A \arrow[rr, "f"] \arrow[dd, equal] && X\vee Y  \arrow[rr, "j\vee 1"] \arrow[dd, "p_1"] && M\vee Y \arrow[dd, "p_1"] \\
        &&&&& \\
        \Sigma A \arrow[rr, "p_1\circ f"] && X \arrow[rr, "j"] && M
    \end{tikzcd}
\end{equation}
where we have \(q=j\vee1\) and \(\varphi=p_1\). Since the map \(\Omega p_1\) has a right homotopy inverse, we may apply Corollary \ref{1cor:splitfib} to the homotopy cofibration in the right-most column of (\ref{1dgm:cofib2}), again giving a homotopy equivalence \(E'\simeq \Omega M\ltimes Y\). Thus Diagram (\ref{1dgm:data2}) becomes
\begin{equation}\label{1dgm:fib1'}
    \begin{tikzcd}[row sep=1.5em, column sep = 1.5em]
        && E \arrow[rr, "\alpha'"] \arrow[dd] && \Omega M\ltimes Y \arrow[dd] \\
        &&&&& \\
        \Sigma A \arrow[rr, "f"] && X\vee Y \arrow[rr, "j\vee 1"] \arrow[dd, "h"] && M\vee Y \arrow[dd, "p_1"] \\
        &&&&& \\
        && M \arrow[rr, equal] && M
    \end{tikzcd}
\end{equation}
and, analogously to (\ref{1cofib1}), there is a homotopy cofibration
\begin{equation*}
    \Omega M\ltimes\Sigma A\xrightarrow{\theta_f'} E\xrightarrow{\alpha'} \Omega M\ltimes Y.
\end{equation*}
Noting that the upper square of (\ref{1dgm:fib1'}) is a homotopy pullback and arguing exactly as in the proof of \cite{t20}*{Lemma 8.3} gives the following.

\begin{lem}\label{1lem:alpha'rightinv}
The map \(\alpha'\) has a right homotopy inverse \(r:\Omega M\ltimes Y\rightarrow E\) such that the composite \(l\circ r\) is null homotopic. \hfill $\square$
\end{lem} 

Combining this with the general situation, we have homotopy cofibrations \[\Omega M\ltimes\Sigma A\xrightarrow{\theta_f} E\xrightarrow{\alpha} E'\text{\; and \;}\Omega M\ltimes\Sigma A\xrightarrow{\theta_f'} E\xrightarrow{\alpha'} \Omega M\ltimes Y.\] By Lemma \ref{1lem:thetaleftinv} there is a left homotopy inverse \(k\) for both $\theta_f$ and $\theta_f'$. Lemma \ref{1lem:alpha'rightinv} gives a right homotopy inverse $r$ for $\alpha'$, and since \(l\circ r\simeq \ast\), Diagram (\ref{1dgm:cofib1}) implies that \(k\circ r\simeq \ast\). By \cite{t20}*{Lemma 8.5}, this implies that the composite 
\begin{equation}\label{1equiv1}
    \Omega M \ltimes Y\xrightarrow{r}E\xrightarrow{\alpha} E'
\end{equation} 
is a homotopy equivalence. This achieves our first goal of gaining more control over the homotopy equivalence \(E'\simeq\Omega M\ltimes Y\); we will use the fact that it factors through \(E\) to prove Theorem \ref{1thm:inertness}. Recall that applying Corollary \ref{1cor:splitfib} to the right-most column of Diagram (\ref{1dgm:data2}) yields a homotopy equivalence \[\Omega C\simeq \Omega M\times \Omega(\Omega M\ltimes Y).\] 

\begin{thmproof}
We have already shown that the map \(\Omega \varphi\) has a right homotopy inverse given by \(\Omega q\circ s\circ t\), due to homotopy commutativity of (\ref{1dgm:setup1}), thus all that remains to prove is that the map \(\Omega q\) also has a right homotopy inverse. We shall do this by showing that the above homotopy equivalence for \(\Omega C\) factors through \(\Omega q\), from which the existence of a right homotopy inverse for \(\Omega q\) follows immediately. Let \(\lambda=s\circ t\), which is a right homotopy inverse for the map \(\Omega h\). Taking loops on the middle and right columns of (\ref{1dgm:data2}) gives 
\begin{equation}\label{1dgm:loopsetup}
    \begin{tikzcd}[row sep=1.5em, column sep = 1.5em]
        \Omega E \arrow[rr, "\Omega \alpha"] \arrow[dd] && \Omega E' \arrow[dd] \\
        &&& \\
        \Omega(X\vee Y) \arrow[rr, "\Omega q"] \arrow[dd, "\Omega h"] && \Omega C \arrow[dd, "\Omega \varphi"] \\
        &&& \\
        \Omega M \arrow[rr, equal] && \Omega M.
    \end{tikzcd}
\end{equation}
Letting \(r'\) denote \(\Omega M\ltimes Y\xrightarrow{r}E\rightarrow X\vee Y\), we have the composite \[e:\Omega M\times \Omega(\Omega M\ltimes Y)\xrightarrow{\lambda\times\Omega r'}\Omega(X\vee Y)\times\Omega(X\vee Y)\xrightarrow{\mu}\Omega(X\vee Y)\xrightarrow{\Omega q}\Omega C\] where \(\mu\) is the loop multiplication. The homotopy commutativity of (\ref{1dgm:loopsetup}) together with the fact that \(\Omega q\) is an \(H\)-map implies that that \(e\) is a homotopy equivalence, so the proof is complete.
\end{thmproof}

\section{Inert Maps and Loop Space Decompositions}\label{sec:inert}

We wish to apply Theorem \ref{1thm:inertness} to the situation of connected sums: suppose we have two homotopy cofibrations of simply connected spaces \[\Sigma A\xrightarrow{f} B\xrightarrow{j} C\text{\; and \;}\Sigma A\xrightarrow{g} D\xrightarrow{l} E.\] Consider the composite $f+g:\Sigma A\xrightarrow{\sigma}\Sigma A\vee \Sigma A \xrightarrow{f\vee g}B\vee D$, where $\sigma$ denotes the suspension co-multiplication on $\Sigma A$. The homotopy cofibre of $f+g$ is called the \textit{generalised connected sum} of $C$ and $E$ over $\Sigma A$, written $C\#_{\Sigma A}E$. When the space $A$ is clear, we will often omit the subscript. 

Note that $p_1\circ(f+g)\simeq f$. We have the diagram below, in which each complete row and column is a homotopy cofibration, and the bottom-right square is a homotopy pushout. We label the induced map $C\#_{\Sigma A}E\rightarrow C$ by $h$.
\begin{equation}\label{dgm:setup}
    \begin{tikzcd}[row sep=1.5em, column sep = 1.5em]
        && D \arrow[rr, equal] \arrow[dd, hook] && D \arrow[dd] \\
        &&&&& \\
        \Sigma A \arrow[rr, "f+g"] \arrow[dd, equal] && B\vee D  \arrow[rr, "q"] \arrow[dd, "p_1"] && C\#_{\Sigma A}E \arrow[dd, "h"] \\
        &&&&& \\
        \Sigma A \arrow[rr, "f"] && B \arrow[rr, "j"] && C.
    \end{tikzcd}
\end{equation}
The following lemma is a stronger version of \cite{chenfib}*{Lemma 2.2}. It follows immediately from Theorem \ref{1thm:inertness} and by applying Corollary \ref{1cor:splitfib} to the rightmost column of Diagram (\ref{dgm:setup}).

\begin{lem} \label{lem:inertness}
Take the setup of Diagram (\ref{dgm:setup}). If the map \(f\) is inert, then so is \(f+g\). Moreover, there is a homotopy equivalence \(\Omega (C\#_{\Sigma A}E)\simeq\Omega C\times\Omega(\Omega C\ltimes D)\). \hfill \(\square\)
\end{lem}

\begin{rem}
Lemma \ref{lem:inertness} shows that whatever the homotopy class of the map \(g\), the map \(f+g\) inherits inertness from \(f\) regardless. 
\end{rem}

We may also consider connected sums of Poincar\'e Duality complexes. Recall that a Poincar\'e Duality complex is a finite, simply connected $CW$-complex whose cohomology ring exhibits Poincar\'e Duality. For such a complex, there exists a cell structure that has a single top-dimensional cell, and we may define the connected sum operation similarly to that of manifolds. Namely, for two $n$-dimensional Poincar\'e Duality complexes $M$ and $N$, the space $M\#N$ is formed by removing an $n$-dimensional open disc from the interior of the top-cells of $M$ and $N$ and joining the resulting complexes along their boundaries. Up to homotopy $M\# N$ coincides with the generalised connected sum $M\#_{S^{n-1}}N$.

In pursuit of our analogue to Wall's Theorem, we seek a framework whereby it may be shown a that Poincar\'e Duality complex has the homotopy type of a connected sum, after looping. To begin to give this we have the following Proposition, which is a restatement of \cite{t20}*{Theorem 9.1 (b)-(c)}, though we provide a new proof.

\begin{prop}[Theriault]\label{prop:PDconnsum}
Let $M$ and $N$ be two Poincar\'e Duality complexes of dimension $n$, where $n>3$, such that the attaching map of the top-cell of $M$ is inert. Then there is a homotopy equivalence \[\Omega (M\#N)\simeq \Omega M\times\Omega(\Omega M\ltimes N_{n-1})\] where $N_{n-1}$ denotes the $(n-1)$-skeleton of $N$. Furthermore, the attaching map of the top-cell of $M\#N$ is inert.
\end{prop}

\begin{proof}
Let $f_M$ and $f_N$ be the attaching maps of the top-cells of $M$ and $N$, respectively, onto their $(n-1)$-skeleta $M_{n-1}$ and $N_{n-1}$. We have homotopy cofibrations \[S^{n-1}\xrightarrow{f_M}M_{n-1}\rightarrow M\text{\; and \;}S^{n-1}\xrightarrow{f_N}N_{n-1}\rightarrow N.\] Similar to Diagram (\ref{dgm:setup}), we have the following homotopy cofibration diagram.
\begin{equation*}
    \begin{tikzcd}[row sep=1.5em, column sep = 1.5em]
        && N_{n-1} \arrow[rr, equal] \arrow[dd, hook] && N_{n-1} \arrow[dd] \\
        &&&&& \\
        S^{n-1} \arrow[rr, "f_M+f_N"] \arrow[dd, equal] && M_{n-1}\vee N_{n-1} \arrow[rr, "q"] \arrow[dd, "p"] && M\#N \arrow[dd, "h"] \\
        &&&&& \\
        S^{n-1} \arrow[rr, "f_M"] && M_{n-1}  \arrow[rr, "j"] && M.
    \end{tikzcd} 
\end{equation*}
The result then follows from Lemma \ref{lem:inertness}.
\end{proof}

So far we have only used Theorem \ref{1thm:inertness} to show that a sum of attaching maps is inert. The next theorem gives conditions for an attaching map to be inert without first supposing that it is a sum. 

\begin{thm} \label{thm:PDconn}
Let $n>3$ and suppose that $M$, $N$ and $P$ are $n$-dimensional Poincar\'e Duality complexes. Let \(f_M\) and \(f_P\) denote the attaching maps of the top-cells of $M$ and $P$, respectively, and suppose further that: 
\begin{enumerate}
    \item[(i)] $M_{n-1}\simeq P_{n-1}\vee N_{n-1}$;
    \item[(ii)] the composite $p_1\circ f_M:S^{n-1}\rightarrow P_{n-1}$ is inert;
    \item[(iii)] the homotopy cofibre of $p_1\circ f_M$, \(Q\), is such that \(\Omega Q\simeq\Omega P\);
    \item[(iv)] the map \(f_P\) is inert.
\end{enumerate}
Then the attaching map $f_M$ is inert and $\Omega M\simeq \Omega(N\#P)$.
\end{thm}

\begin{proof}
From \textit{(i)} we have the following homotopy cofibration diagram
\begin{equation}\label{dgm:PDconn}
    \begin{tikzcd}[row sep=3em, column sep = 3em]
        & N_{n-1} \arrow[r, equal] \arrow[d, hook] & N_{n-1}  \arrow[d] \\
        S^{n-1} \arrow[r, "f_M"] \arrow[d, equal] & P_{n-1}\vee N_{n-1} \arrow[r] \arrow[d, "p_1"] & M \arrow[d] \\
        S^{n-1} \arrow[r, "p_1\circ f_M"] & P_{n-1} \arrow[r] & Q.
    \end{tikzcd}
\end{equation}
Condition \textit{(ii)} places us in the situation of Theorem \ref{1thm:inertness}, therefore \(f_M\) is inert and there is a homotopy equivalence \(\Omega M\simeq\Omega Q\times \Omega(\Omega Q\ltimes N_{n-1})\). By condition \textit{(iii)}, we therefore have
\begin{equation}\label{eqn:generalthm}
    \Omega M\simeq\Omega P\times \Omega(\Omega P\ltimes N_{n-1}).
\end{equation}
Now consider the connected sum $N\#P$. There is a homotopy cofibration diagram
\begin{equation*}
    \begin{tikzcd}[row sep=3em, column sep = 3em]
        & N_{n-1} \arrow[r, equal] \arrow[d, hook] & N_{n-1}  \arrow[d] \\
        S^{n-1} \arrow[r, "f_N+f_P"] \arrow[d, equal] & P_{n-1}\vee N_{n-1} \arrow[r] \arrow[d, "p_1"] & N\#P \arrow[d] \\
        S^{n-1} \arrow[r, "f_P"] & P_{n-1} \arrow[r] & P.
    \end{tikzcd}
\end{equation*}
By \textit{(iv)}, Lemma \ref{lem:inertness} gives a homotopy equivalence \(\Omega(N\#P)\simeq\Omega P\times \Omega(\Omega P\ltimes N_{n-1})\) and consequently that \(\Omega M\simeq \Omega(N\#P)\), due to (\ref{eqn:generalthm}).
\end{proof}

This provides a neat corollary involving connected sums of products of spheres, which we form by drawing on results of Beben-Theriault from \cite{bt1}. 

\begin{cor} \label{cor:Sconn}
Let $n,m\geq2$ be integers. Suppose that $M$ and $N$ are two $(n+m)$-dimensional Poincar\'e Duality complexes such that there is a cohomology isomorphism $H^*(M)\cong H^*(N\#(S^n\times S^m))$ and there exists a homotopy equivalence \[M_{n+m-1}\simeq S^n\vee S^m\vee N_{n+m-1}.\] Then there is a homotopy equivalence $\Omega M\simeq \Omega(N\#(S^n\times S^m))$ and the attaching map of the top-cell of $M$ is inert.
\end{cor}

\begin{proof}
There is a homotopy commutative diagram of homotopy cofibrations
\begin{equation}\label{dgm:Sconn}
    \begin{tikzcd}[row sep=3em, column sep = 3em]
        & N_{n+m-1} \arrow[r, equal] \arrow[d, hook, "i"] & N_{n+m-1}  \arrow[d, "j\circ i"] \\
        S^{n+m-1} \arrow[r, "f_M"] \arrow[d, equal] & S^n\vee S^m\vee N_{n+m-1} \arrow[r, "j"] \arrow[d, "p"] & M \arrow[d] \\
        S^{n+m-1} \arrow[r, "p\circ f_M"] & S^n\vee S^m \arrow[r, "h"] & Q.
    \end{tikzcd}
\end{equation}
Because of the conditions we imposed in the statement of the Corollary, the complex \(Q\) precisely satisfies the situation of \cite{bt1}*{Lemma 2.3}, so there is a homotopy equivalence $\Omega Q\simeq\Omega S^n\times\Omega S^m$ and the map $\Omega h$ has a right homotopy inverse. Therefore the composite \(p\circ f_M\) is inert (by definition), so Theorem \ref{thm:PDconn} applies and there is a homotopy equivalence $\Omega M\simeq \Omega(N\#(S^n\times S^m))$ and \(f_M\) is inert.
\end{proof}

\begin{exa}
An immediate application of Corollary \ref{cor:Sconn} is that the attaching map of the top-cell of the \(r\)-fold connected sum \[\overset{r}{\underset{i=1}{\#}}(S^n\times S^n)\] is inert. We can generalise this further: fix an integer \(k>3\) and take an index set $I$ and integers $n_i,m_i\geq2$ such that $n_i+m_i=k$ for all $i\in I$. The attaching map of the top-cell of the connected sum $\underset{i\in I}{\#}(S^{n_i}\times S^{m_i})$ is then easily seen to be inert.
\end{exa}

\begin{exa} \label{exa:bt}
Using Corollary \ref{cor:Sconn} we can deduce some facts about a manifolds discussed in \cite{bt1}. Let \(n>3\) be an integer such that \(n\neq4,8\) (to avoid cases where maps of Hopf invariant one may arise) and let \(M\) be an a smooth, closed, oriented, \((n-1)\)-connected \(2n\)-dimensional manifold with \(rank(H_n(M))=d\geq2\). The manifold \(M\) is Poincar\'e Duality complex, and we are in the situation of Corollary \ref{cor:Sconn}. Consequently, all such manifolds have their top-cells attached by inert maps. Moreover, if \(n\) is an odd number, Poincar\'e Duality implies that \(d\) must be even, and so \cite{bt1}*{Theorem 1.1(b)} implies that there is a homotopy equivalence \[\Omega M\simeq\Omega\left(\overset{d/2}{\underset{i=1}{\#}}(S^n\times S^n)\right)\] and consequently an isomorphism of homotopy groups. Thus, for such a manifold \(M\), its homotopy groups are determined entirely by the rank of its middle (co)homology group \(H_n(M)\).
\end{exa}

\section{Concerning $(n-2)$-Connected $2n$-Dimensional Poincar\'e Duality Complexes}

Let us fix an integer $n>3$ such that $n\not\in\lbrace4,8\rbrace$ (again to avoid cases where maps of Hopf invariant one may arise) and consider a  smooth, closed, oriented, $(n-2)$-connected $2n$-dimensional Poincar\'e Duality complex. Take such a complex \(M\); the Universal Coefficient Theorem together with Poincar\'e Duality enables us to deduce that 
\begin{equation}\label{eqn:hmlgy}
    H_*(M)\cong\begin{cases}\Z\mbox{ if }*=0\mbox{ or }*=2n \\ \Z^{l}\oplus T\mbox{ if }*=n-1 \\\Z^{d}\oplus T\mbox{ if }*=n \\ \Z^{l}\mbox{ if }*=n+1 \\ 0\mbox{ otherwise}\end{cases}
\end{equation}
where $T\cong\bigoplus_{i=1}^k\Z/p_i^{r_i}\Z$ for primes $p_i$ and integers $r_i\in\N$. In this section we shall construct an appropriate homology decomposition of $M$, so that the homotopy theoretic methods we have developed previously can be applied. This will provide the basis for proof of the analogue to Wall's Theorem (see Theorem \ref{thm:wallanalogue}).  

To begin, recall from \cite{hatcher}*{Section 4H} (or \cite{ark}*{Section 7.3}) that for a sequence of groups $G_j$, $j\geq 1$, one may inductively construct a $CW$-complex $X$ via a sequence of subcomplexes $X_1\subset X_2\subset\dots$ with \[H_i(X_j)\cong\begin{cases} G_i\mbox{ if }i\leq j \\ 0\mbox{ if }i>j\end{cases}\] such that
\begin{itemize}
    \item[(i)] $X_1$ is a Moore space $M(G_1,1)$;
    \item[(ii)] $X_{j+1}$ is the homotopy cofibre of a cellular map $h_j:M(G_{j+1},j)\rightarrow X_j$ that induces a trivial map in homology;
    \item[(iii)] $X=\bigcup_j X_j$.
\end{itemize}
In \cite{hatcher}*{Theorem 4H.3}, it is noted that every simply connected $CW$-complex has a homology decomposition. We will follow established convention and let $P^n(p_i^{r_i})$ denote the Moore space \(M(\Z/p_i^{r_i}\Z,n-1)\), that is, the homotopy cofibre of the degree $p_i^{r_i}$ map $S^{n-1}\rightarrow S^{n-1}$.

\begin{prop} \label{prop:pdhmlgy}
Let $M$ be an $(n-2)$-connected $2n$-dimensional Poincar\'e Duality complex, as described in (\ref{eqn:hmlgy}). Then there exists an integer $c$, with $0\leq c\leq l$, such that there is a homotopy cofibration
\[S^{2n-1}\rightarrow\left(\bigvee_{i=1}^{d} S^n\right)\vee\left(\bigvee_{i=1}^{l-c} S^{n-1}\vee S^{n+1}\right)\vee J\rightarrow M\] for some appropriate $CW$-complex $J$.
\end{prop}

The proof of Proposition \ref{prop:pdhmlgy} will involve two elementary lemmas regarding homotopy cofibrations and wedge sums, which we record here for ease of reference.

\begin{letlem}\label{lem:known1}
Let \(A\vee B\xrightarrow{h}C\rightarrow D\) be a homotopy cofibration. If the restriction of the map \(h\) top \(A\) is null homotopic, then \(\Sigma A\) is a homotopy retract of \(D\). \hfill \(\square\)
\end{letlem}

\begin{letlem}\label{lem:known2}
Let \(A\xrightarrow{h}B\vee C\rightarrow D\) be a homotopy cofibration. If the composition of $h$ with the pinch map $p_1:B\vee C\rightarrow B$ is null homotopic, then $B$ is a homotopy retract of $D$. \hfill \(\square\)
\end{letlem}

\begin{propproof}
To give the asserted homotopy cofibration, we will construct a homology decomposition $M$. We start with \[M_1\simeq M_2\simeq\dots\simeq M_{n-2}\simeq *\] and then \[M_{n-1}\simeq \left(\bigvee_{i=1}^l S^{n-1}\right)\vee\left(\bigvee_{i=1}^k P^n(p_i^{r_i})\right).\] The next complex $M_n$ is constructed via the homotopy cofibration \[M(\Z^d\oplus T,n-1)\xrightarrow{h_{n-1}}M_{n-1}\rightarrow M_n\]
for which we shall take \[M(\Z^d\oplus T,n-1)=\left(\bigvee_{i=1}^{d} S^{n-1}\right)\vee\left(\bigvee_{i=1}^k P^n(p_i^{r_i})\right).\] Since $(h_{n-1})_*=0$, the Hurewicz Theorem implies that the restriction of $h_{n-1}$ to the wedge $\bigvee_{i=1}^{d} S^{n-1}$ is null homotopic. By Lemma \ref{lem:known1}, we therefore have that the suspension of the wedge $\bigvee_{i=1}^{d} S^{n-1}$ retracts off $M_n$. Hence we have a homotopy equivalence \[M_n\simeq\bigvee_{i=1}^{d} S^n\vee E\] where $E$ denotes the homotopy cofibre of the inclusion $\bigvee_{i=1}^{d} S^n \hookrightarrow M_n$. 

We now deduce some further information about the homotopy type of the complex $E$. Since $h_{n-1}$ induces a trivial map in homology, we may suppose that there exists an integer $c_1$, with $0\leq c_1\leq l$, such that the composite \[\left(\bigvee_{i=1}^{2d} S^{n-1}\right)\vee\left(\bigvee_{i=1}^k P^n(p_i^{r_i})\right)\xrightarrow{h_{n-1}}M_{n-1}\simeq\left(\bigvee_{i=1}^l S^{n-1}\right)\vee\left(\bigvee_{i=1}^k P^n(p_i^{r_i})\right) \xrightarrow{p} \bigvee_{i=1}^{l-c_1} S^{n-1}\] is null homotopic, where $p$ denotes a pinch map. Moreover, without loss of generality we may suppose that $c_1$ is the minimal integer with this property. Here we use Lemma \ref{lem:known2}, and thus we may write \[E\simeq \left(\bigvee_{i=1}^{l-c_1} S^{n-1}\right)\vee J_1\] for some $CW$-complex \(J_1\). The next step in our construction is to form $M_{n+1}$ via the homotopy cofibration \[M(\Z^l,n)=\bigvee_{i=1}^l S^n\xrightarrow{h_n}M_n\rightarrow M_{n+1}.\] First, note that the composite \[\bigvee_{i=1}^l S^n\xrightarrow{h_n}M_n\simeq\left(\bigvee_{i=1}^{d} S^n\right)\vee\left(\bigvee_{i=1}^{l-c_1} S^{n-1}\right)\vee J_1 \xrightarrow{p_1} \bigvee_{i=1}^{d} S^n\] is forced to be null homotopic, otherwise we would generate additional torsion in $M_{n+1}$, which is not permissible because $H_{n+1}(M)\cong\Z^l$. Therefore, again by Lemma \ref{lem:known1}, $\bigvee_{i=1}^{d}S^n$ is a homotopy retract of $M_{n+1}$. Furthermore, let $c_2\leq c_1$ be the smallest integer such that the composite \[\bigvee_{i=1}^l S^n\xrightarrow{h_n}M_n \xrightarrow{p} \bigvee_{i=1}^{l-c_2} S^n\subseteq\bigvee_{i=1}^{l-c_1} S^n\] is null homotopic. Such a \(c_2\geq0\) exists, and we may assume it is minimal without loss of generality.  

Let $c_3$ be the least integer, $0\leq c_3\leq l$, such that the restiction of $h_n$ to the sub-wedge $\bigvee_{i=1}^{l-c_3}S^n$ is null homotopic. Once again using Lemma \ref{lem:known2}, this gives rise to a homotopy equivalence \[M_{n+1}\simeq\left(\bigvee_{i=1}^{d} S^n\right)\vee\left(\bigvee_{i=1}^{l-c_2} S^{n-1}\right)\vee \left(\bigvee_{i=1}^{l-c_3}S^{n+1}\right)\vee J_2\]  for some other $CW$-complex \(J_2\). Rewriting this, letting $c=Max\lbrace c_2,c_3\rbrace$, we have \[M_{n+1}\simeq\left(\bigvee_{i=1}^{d} S^n\right)\vee\left(\bigvee_{i=1}^{l-c} S^{n-1}\vee S^{n+1}\right)\vee J\] where \(J\) arises from taking the wedge of \(J_2\) with the discarded spheres. We call the \(CW\)-complex $J$ the \textit{auxiliary complex}. Letting $f$ denote the attaching map of the top-cell, our Poincar\'e Duality complex $M$ is then given by the homotopy cofibration \[S^{2n-1}\xrightarrow{f}\left(\bigvee_{i=1}^{d} S^n\right)\vee\left(\bigvee_{i=1}^{l-c} S^{n-1}\vee S^{n+1}\right)\vee J\xrightarrow{j}M.\]
\end{propproof}

\begin{exa}
In the general case, we deliberately leave the homotopy type of the auxiliary complex mysterious, but there are circumstances in which we may deduce its homotopy type. If $H_*(M)$ is torsion free and $d=0$, we can draw upon a result of Huang \cite{huanggauge}. These conditions demand that $H_n(M)\cong0$, and therefore there are no cells of consecutive dimensions in the $CW$-structure of $M$. Thus there is a homotopy cofibration \[\bigvee_{i=1}^l S^n\xrightarrow{\psi}\bigvee_{i=1}^l S^{n-1}\rightarrow M_{n+1}\] that defines the $(n+1)$-skeleton of $M$. Further, letting $i_r$ denote the inclusion of the $r^{th}$ wedge summand, for each $r\in\lbrace1,\dots,l\rbrace$ the composite \[\psi_r:S^n\xhookrightarrow{i_r}\bigvee_{i=1}^l S^n\xrightarrow{\varphi}\bigvee_{i=1}^l S^{n-1}\] defines a homotopy class in the group $\pi_n(\bigvee_{i=1}^l S^{n-1})$. Since $n>3$, this group is isomorphic to $\bigoplus_{i=1}^l\Z/2\Z$ \cite{toda}, where each $\Z/2\Z$ summand is generated by the homotopy class of the attaching map for the $(n+1)$-cell of $\Sigma^{n-3}\C P^2$. Thus we may from an $(l\times l)$-matrix $C$ with entries in $\Z/2\Z$, where the $r^{th}$ column is the image of the homotopy class of $\psi_r$ under this group isomorphism. Huang shows that this matrix may be manipulated by row and column operations, and that these operations are homotopy invariant. Letting $c=rank(C)$, \cite{huanggauge}*{Lemma 6.1} shows that there exists a homotopy equivalence \[M_{n+1}\simeq\left(\bigvee_{i=1}^{l-c}(S^{n-1}\vee S^{n+1})\right)\vee\left(\bigvee_{i=1}^c \Sigma^{n-3}\C P^2\right).\] If $d>0$, then because of our torsion free assumption, maps between cells of consecutive dimension must be null homotopic, so we have \[M_{n+1}\simeq\left(\bigvee_{i=1}^{d}S^n\right)\vee\left(\bigvee_{i=1}^{l-c}(S^{n-1}\vee S^{n+1})\right)\vee\left(\bigvee_{i=1}^c \Sigma^{n-3}\C P^2\right).\] 
\end{exa}

Our aim in the construction of this section was to write the skeletal structure of $M$ in such a fashion as to have as many wedges of pairs of spheres retracting off it. This is key to our approach to the Main Theorem, and to sustaining the analogy with Wall's Theorem. The decomposition of Proposition \ref{prop:pdhmlgy} enables us to make the following observation about the cohomology of the complex \(M\). As \(M\) is a Poincar\'e Duality complex, it has a fundamental class, which we will denote by \(\mu_M\).

\begin{lem}\label{lem:spherecups}
Assume \(d>1\) and let $x\in H^n(M)$ be a generator induced by an \(S^n\) wedge summand in $M_{n+1}$. Then there is a class $y\in H^n(M)$, is induced a different \(S^n\) wedge summand, such that $x\cup y=\mu_M$.
\end{lem}

\begin{proof}
The given class \(x\) associated with an \(S^n\) wedge summand in \(M_{n+1}\) is a basis element in the cohomology group \(H^n(M)\). By \cite{hatcher}*{Corollary 3.39} there exists a class \(y\in H^n(M)\) that generates an infinite cyclic summand of \(H^n(M)\), such that $x\cup y=\mu_M$. Therefore, up to a change of basis of \(H^*(M)\) (i.e. up to a self equivalence of \(M_{n+1}\)) we have that the class \(y\) is also induced by an \(S^n\) wedge summand. Furthermore, because we excluded Hopf invariant one cases in the setup of this section, we have \(\pm x\neq\pm y\), so the spheres that induce the classes \(x\) and \(y\) are distinct.
\end{proof}

\section{Proving the Analogue}\label{sec:wall}

We now apply the methods we have developed to give the titular homotopy theoretic analogue. Recall from the previous section that for a smooth, closed, oriented, \((n-2)\)-connected \(2n\)-dimensional manifold Poincar\'e Duality complex \(M\), with $n>3$ such that $n\not\in\lbrace4,8\rbrace$, we have integral homology \[H_*(M)\cong\begin{cases}\Z\mbox{ if }*=0\mbox{ or }*=2n \\ \Z^{l}\oplus T\mbox{ if }*=n-1 \\\Z^{d}\oplus T\mbox{ if }*=n \\ \Z^{l}\mbox{ if }*=n+1 \\ 0\mbox{ otherwise}\end{cases}\] where $T\cong\bigoplus_{i=1}^k\Z/p_i^{r_i}\Z$ for primes $p_i$ and integers $r_i\in\N$. 

\begin{thm}\label{thm:wallanalogue}
Let $n>3$ be an integer such that $n\not\in\lbrace4,8\rbrace$, and let $M$ be a \((n-2)\)-connected \(2n\)-dimensional Poincar\'e Duality complex with \(d>1\). Then there exists a homotopy equivalence \[\Omega M\simeq\Omega(M_1\#M_2\#M_3)\] where
\begin{itemize}
    \item[(i)] \(M_1\) is an \((n-1)\)-connected \(2n\)-dimensional Poincar\'e Duality complex, with \(rank(H_n(M_1))=d\);
    \item[(ii)] \(M_2\) is a connected sum of finitely many copies of \(S^{n-1}\times S^{n+1}\) and;
    \item[(iii)] $M_3$ is a \(CW\)-complex with $H_n(M_3)$ finite.
\end{itemize}
\end{thm}

\begin{proof}
We first produce a loop space decomposition for \(M_1\#M_2\#M_3\). In general, we take \(M_1\) to be as in \textit{(i)} above, but note that if \(d\) is even we may simply take \(M_1\) to be a connected sum of \(\frac{d}{2}\)-many copies of \(S^n\times S^n\). Letting \(c\) and \(J\) be as in Proposition \ref{prop:pdhmlgy}, we define \[M_2=\overset{l-c}{\underset{i=1}{\#}}(S^{n-1}\times S^{n+1})\text{ and }M_3=J\cup e^{2n}.\] Note that we neglect to denote the map by which we attach the \(2n\)-cell to \(J\), as its homotopy type is of no consequence. For brevity, let \[W_1=\bigvee_{i=1}^{d} S^n\text{, }W_2=\bigvee_{i=1}^{l-c}(S^{n-1}\vee S^{n+1})\text{\; and \;}X=\left(\bigvee_{i=1}^{d-2}S^n\right)\vee W_2\vee J\] so by construction \((M_1)_{2n-1}\simeq W_1\) and \((M_2)_{2n-1}\simeq W_2\). By Lemma \ref{lem:spherecups} we can always isolate two \(n\)-spheres in \(W_1\) that are associated with classes that cup together to give the fundamental class, thus giving rise to the homotopy cofibration diagram
\begin{equation}\label{dgm:wall1}
    \begin{tikzcd}[row sep=3em, column sep=3em]
        & X \arrow[r, equal] \arrow[d, hook, "i"] & X \arrow[d] \\
        S^{2n-1} \arrow[r] \arrow[d, equal] & W_1\vee W_2 \vee J\arrow[r] \arrow[d, "p"] & M_1\#M_2\#M_3 \arrow[d, "h"] \\
        S^{2n-1} \arrow[r] & S^n\vee S^n \arrow[r, "q"] & Q
    \end{tikzcd}
\end{equation}
where the space \(Q\) has the property that \(H^*(Q)\cong H^*(S^n\times S^n)\). By \cite{bt1}*{Lemma 2.3}, the map \(\Omega q\) has a right homotopy inverse and \(\Omega Q\simeq \Omega (S^n\times S^n)\). By homotopy commutativity of (\ref{dgm:wall1}) the map \(\Omega h\) therefore has a right homotopy inverse, and applying Corollary \ref{1cor:splitfib} to the right-hand column gives the loop space decomposition \[\Omega(M_1\#M_2\#M_3)\simeq \Omega(S^n\times S^n)\times\Omega(\Omega(S^n\times S^n)\ltimes X).\] 

Now consider the Poincar\'e Duality complex \(M\). Similarly to above, letting $f$ denote the attaching map of the top-cell of $M$, by Proposition \ref{prop:pdhmlgy} and Lemma \ref{lem:spherecups} we have the following diagram of homotopy cofibrations
\begin{equation}\label{dgm:wall2}
    \begin{tikzcd}[row sep=3em, column sep=3em]
        & X \arrow[r, equal] \arrow[d, hook, "i"] & X \arrow[d, "j\circ i"] \\
        S^{2n-1} \arrow[r, "f"] \arrow[d, equal] & W_1\vee W_2 \vee J\arrow[r, "j"] \arrow[d, "p"] & M \arrow[d, "h"] \\
        S^{2n-1} \arrow[r, "p\circ f"] & S^n\vee S^n \arrow[r, "q"] & Q'
    \end{tikzcd}
\end{equation}
where, again, \(Q'\) is such that \(H^*(Q)\cong H^*(S^n\times S^n)\). Reasoning identically to before gives the loop space decomposition \[\Omega M\simeq \Omega(S^n\times S^n)\times\Omega(\Omega(S^n\times S^n)\ltimes X).\] Comparing this to the decomposition for \(\Omega(M_1\# M_2\# M_3)\) gives us the desired homotopy equivalence. 
\end{proof}

The proof of Theorem \ref{thm:wallanalogue} shows that after we take loop spaces, the homotopy class of the attaching map of the top-cell ceases to be important. Indeed, the principal object of concern in the proof is the $(2n-1)$-skeleton of the complex $M$, and the connected sum $M_1\#M_2\#M_3$ is constructed so that its $(2n-1)$-skeleton exactly matches that of $M$. Observe also that Theorem \ref{thm:wallanalogue} also gives inertness of the attaching map of the top-cell of $M$, by application of Proposition \ref{prop:PDconnsum} and Example \ref{exa:bt}.

\begin{rem}
Proving Theorem \ref{thm:wallanalogue} for Poincar\'e Duality complexes, as we have done, implies we have such a composition in the case when $M$ is in fact an smooth, closed, oriented, $(n-2)$-connected $2n$-manifold. In that case, the complex $M_3$ may also have the homotopy type of a manifold - this depends on whether the \textit{total surgery obstruction} of Ranicki (see \cite{ranickibook}*{Theorem 17.4}) is zero. It would be interesting to make further investigation here.
\end{rem}

Theorem \ref{thm:wallanalogue} enables us to make a further observation regarding rational hyperbolicity. Indeed, recall that a simply connected space $Y$ is called \textit{rationally elliptic} if $dim(\pi_*(Y)\otimes\Q)<\infty$, and called\textit{ rationally hyperbolic} otherwise \cite{fht}. For example, any wedge of spheres \(\bigvee_{i=1}^r S^{m_i}\) with \(r>1\) and all \(m_i>1\) is a rationally hyperbolic space.

\begin{cor}\label{cor:hyp}
Let $n>3$ be an integer such that $n\not\in\lbrace4,8\rbrace$, and let $M$ be a \((n-2)\)-connected \(2n\)-dimensional Poincar\'e Duality complex with \(d>1\). Then \(M\) is rationally hyperbolic if and only if \(d>2\) or \(H_{n-1}(M)\not\cong0\).
\end{cor}

\begin{proof}
Recall from the proof of Theorem \ref{thm:wallanalogue} that we had \[\Omega M\simeq\Omega(S^n\times S^n)\times\Omega(\Omega(S^n\times S^n)\ltimes X)\] where \(X=\left(\bigvee_{i=1}^{d-2}S^n\right)\vee W_2\vee J\). By assuming \(d>2\) or \(H_{n-1}(M)\not\cong0\), the construction of Proposition \ref{prop:pdhmlgy} guarantees that \(X\) does not have the homotopy type of a point. More than this, as \(X\) is an \((n-2)\)-connected \((n+1)\)-dimensional \(CW\)-complex, since our restrictions on \(n\) give \(n\geq5\), we are able to invoke \cite{ganeacogroups} and show that \(X\) in fact has the homotopy type of suspension. Let us write \(X\simeq\Sigma X'\). 

Thus we have \(\Omega(S^n\times S^n)\ltimes \Sigma X'\simeq\Sigma(\Omega(S^n\times S^n)\wedge X')\vee X')\). Rationally, a suspension is homotopy equivalent to a wedge of spheres, so there is a rational homotopy equivalence \[\Omega(S^n\times S^n)\ltimes \Sigma X'\simeq \bigvee_{i=1}^r S^{m_i}\] for some integers \(m_i>1\) and \(r>1\). Therefore, rationally, \(\Omega(\bigvee_{i=1}^r S^{m_i})\) retracts off \(\Omega M\), and \(M\) is consequently rationally hyperbolic.

We prove the other direction by negation: if \(d=2\) and \(H_{n-1}(M)\cong0\), Theorem \ref{thm:wallanalogue} implies that \(M\) has the loop space homotopy type of \(\Omega(S^n\times S^n)\), and so \[\pi_*(M)\cong\pi_*(S^n)\times\pi_*(S^n).\] The complex \(M\) is therefore rationally elliptic, and in particular, not rationally hyperbolic. 
\end{proof}

\bibliographystyle{amsplain}
\bibliography{bib}

\end{document}